\documentclass{amsart}

\usepackage{amsmath}
\usepackage{amssymb}
\usepackage{hyperref}
\usepackage{enumitem}
\usepackage{mathrsfs}

\newtheorem{theorem}{Theorem}[section]
\newtheorem{lemma}[theorem]{Lemma}
\newtheorem{proposition}[theorem]{Proposition}

\theoremstyle{definition}

\theoremstyle{remark}
\newtheorem{remark}[theorem]{Remark}

\numberwithin{equation}{section}

\begin{document}

\title[Restriction of eigenfunctions on products of spheres]{Restriction of eigenfunctions on products of spheres to submanifolds of maximal flats}

\author{Yunfeng Zhang}
\address{Department of Mathematical Sciences, University of Cincinnati, Cincinnati, OH 45221-0025}
\email{zhang8y7@ucmail.uc.edu}






\begin{abstract}
Let $M$ be a product of rank-one symmetric spaces of compact type, each of dimension at least $3$. We establish sharp $L^p$ bounds 
for the restriction of Laplace--Beltrami eigenfunctions on $M$ to arbitrary submanifolds contained in a maximal flat, for all $p \ge 2$. 
The proof combines precise asymptotics of Jacobi polynomials and positivity of Fourier coefficients of spherical functions.
\end{abstract}

\maketitle

\section{Introduction}
Let $M$ be a compact Riemannian manifold of dimension $d$. Let $\Delta$ be the Laplace--Beltrami operator on $M$, and let $f$ be its eigenfunctions such that $\Delta f=-N^2f$, $N>1$. The study of concentration of $f$ has become a popular subject \cite{Zel17}, and one of the approaches is to bound its restriction to submanifolds of $M$, as initiated by Tataru \cite{Tat98}  and Reznikov \cite{Rez04}. Let $S$ be a submanifold of dimension $k$. Let $L^p(M)$, $L^p(S)$ be the Lebesgue spaces associated to the Riemannian volume measure on $M$ and $S$ respectively. The following fundamental restriction bounds for general $M$ were established by Burq--G\'erard--Tzvetkov \cite{BGT07} and Hu \cite{Hu09}: it holds
\begin{align}\label{eq: BGT}
    \|f\|_{L^p(S)}\lesssim N^{\rho(k,d,p)}\|f\|_{L^2(M)},
\end{align}
where 
\begin{align*}
\rho(d-1,d,p)&=\left\{\begin{array}{ll}
\frac{d-1}{2}-\frac{d-1}{p}, \qquad \text{if }\frac{2d}{d-1}\leq p\leq\infty,\\
\frac{d-1}{4}-\frac{d-2}{2p}, \qquad \text{if }2\leq p\leq \frac{2d}{d-1},
\end{array}\right.\\
\rho(d-2,d,p)&=\frac{d-1}{2}-\frac{d-2}{p},  \qquad  \text{if }2<p\leq \infty, \\
\rho(k,d,p)&=\frac{d-1}{2}-\frac{k}{p}, \qquad  \text{if }1\leq k\leq d-3\text{ and }2\leq p\leq \infty;
\end{align*}
and for the case $p=2$ and $k=d-2$, it holds 
\begin{align}\label{eq: logloss}
\|f\|_{L^2(S)}\lesssim N^{\frac{1}{2}}(\log N)^{\frac{1}{2}}\|f\|_{L^2(M)}.
\end{align}
All the above estimates are sharp on spheres, except for the logarithmic loss, which however has been eliminated by Chen--Sogge \cite{CS14} and Wang--Zhang \cite{WZ21} for totally geodesic submanifolds.

The goal of this note is to record an instance of \textit{sharp polynomial improvement} over the above results of Burq--G\'erard--Tzvetkov and Hu, on products of compact rank-one symmetric spaces, for all $p\geq 2$: namely, the estimate \eqref{eq: no loss} below. For each $i=1,\ldots,r$ ($r\geq 2$), let $M_i$ be a compact rank-one symmetric space (CROSS) (of compact type) of dimension $d_i$; this family includes the standard spheres. Let $M=M_1\times \cdots \times M_r$ be their product. A maximal flat of $M$ is a maximally embedded flat, totally geodesic submanifold, which is isomorphic to $\mathbb{T}^r$. We establish:

\begin{theorem}\label{thm: prod}
Let $S$ be a $k$-dimensional submanifold of a maximal flat of $M$, $k=0,1,\ldots,r$. 
Suppose without loss of generality that $d_1\leq d_2\leq\cdots\leq d_r$. 
Let 
$$\tau(d,p):=\left\{\begin{array}{ll}
-\frac1p, & p\geq\frac{4}{d-1},\\
-\frac{d-1}{4}, & 0<p\leq\frac{4}{d-1}.
\end{array}\right.$$
Let $f$ be an eigenfunction of $\Delta$ such that $\Delta f=-N^2f$, $N>1$. Then
\begin{align}\label{eq: general}
    \|f\|_{L^p(S)}\lesssim_\varepsilon N^{\frac{d-2}{2}+\sum_{i=1}^k\tau(d_i,p)+\varepsilon} \|f\|_{L^2(M)}, \qquad\text{for all }p\geq 2.
\end{align}
Moreover, when $r \ge 5$, the $\varepsilon$-loss in \eqref{eq: general} can be removed.
In particular, if $\dim M_i\geq 3$ for all $i$ and $r\geq 5$,  
then
\begin{align}\label{eq: no loss}
    \|f\|_{L^p(S)}\lesssim N^{\frac{d-2}{2}-\frac{k}{p}}\|f\|_{L^2(M)}, \qquad\text{for all }p\geq 2, 
\end{align}
and the above estimate is sharp. 
\end{theorem}

\begin{remark}
    In the setting of compact Lie groups, the author established the same sharp estimate as \eqref{eq: no loss} for eigenfunctions restricted to submanifolds of maximal flats  \cite{Zha26}. 
\end{remark}

\begin{remark}
In the setting of standard tori, Huang--Zhang proved that for totally geodesic submanifolds defined by linear equations with rational coefficients, the same eigenfunction restriction estimate as \eqref{eq: no loss} holds for all $p \geq 2$, up to an $\varepsilon$-loss \cite{HZ21}.
\end{remark}

\begin{remark}
In particular, our estimates apply to geodesics, since every geodesic lies in a maximal flat of $M$. 
It remains an interesting question to get sharp estimates for restriction of eigenfunctions to submanifolds which are not contained in any maximal flat. A possible approach, which we leave for future work, is to adapt the argument of Burq--G\'erard--Tzvetkov \cite{BGT07} to the product manifold setting. 
\end{remark}

\begin{remark}
When two-dimensional factors such as $\mathbb{S}^2$ or $\mathbb{RP}^2$ are present, the bound \eqref{eq: general} is not expected to be sharp even without the $\varepsilon$-loss, as can be seen through the case of products of two-spheres, for which the estimate \eqref{eq: general} when specialized to the case $p=2$ is worse than the general bounds \eqref{eq: BGT} of Burq--G\'erard--Tzvetkov and Hu. It is thus an interesting question to get sharp bounds for this bad case. 
\end{remark}

\begin{remark}
Suppose $S$ is a submanifold of $M$ of product type, namely, $S=S_1\times\cdots \times S_r$ where $S_i$ is a 
submanifold of $M_i$ for each $i$. Then by applying to each $(M_i,S_i)$ the estimate \eqref{eq: BGT}, along with the lossless version of \eqref{eq: logloss} in the cases where it is available, one can recover Theorem \ref{thm: prod} in this product-type setting. Thus the purpose of Theorem \ref{thm: prod} is to treat submanifolds which are not of product type. 
\end{remark}

\begin{remark}
For global $L^p$ estimates of eigenfunctions, the author established on general symmetric spaces of compact type that 
    $$\|f\|_{L^p(M)}\lesssim N^{\frac{d-2}{2}-\frac{d}{p}}\|f\|_{L^2(M)},$$
for $r\geq 5$ and $p>2+{8}/{(r-4)}$  \cite{Zha21}. 
The above sharp estimate also holds on products of CROSSs for $r\geq 5$ and $p\geq {2(d_i+1)}/{(d_i-1)}$ for all $i$, by applying Sogge's classical spectral projector estimates \cite{Sog88} to each factor. It remains open to establish sharp $L^p$ estimates for all $p\geq 2$. 
\end{remark}

Our approach to proving Theorem \ref{thm: prod} is to extract information about general eigenfunctions from a special class of eigenfunctions, namely the spherical functions. 
A key ingredient is the mapping property of convolution with Jacobi polynomials, which express spherical functions on CROSSs. Let $P_n^{(\alpha,\beta)}(x)$ ($n\in\mathbb{Z}_{\geq 0}$) denote the Jacobi polynomials, defined as the orthogonal polynomials on the interval $[-1,1]$ with the weight function $(1-x)^\alpha(1+x)^\beta$. 
We follow the classical treatment in Szeg\"o's book \cite{Sze75}. In particular, we use the same normalization 
$$P_n^{(\alpha,\beta)}(1)=\begin{pmatrix}
    n+\alpha\\ n
\end{pmatrix}.$$
It will be convenient to denote, for $\delta\geq 0$, $p>0$, and $n\geq 0$,
\begin{align}\label{eq: A(delta,p,n)}
    A(\delta,p,n):=\left\{\begin{array}{ll}
      (n+1)^{\delta-\frac{1}{p}},  & p>\frac{1}{\delta+\frac12}, \\
         (n+1)^{-\frac{1}{2}},  & 0<p\leq \frac{1}{\delta+\frac12}.
   \end{array}\right. 
\end{align}
We establish: 

\begin{theorem}\label{thm: jac}
Assume that $0\leq\beta\leq \alpha$. For $n\in\mathbb{Z}_{\geq 0}$, let $T_{n}^{\alpha,\beta}$
   denote the convolution operator on $\mathbb{T}=\mathbb{R}/2\pi\mathbb{Z}$ with kernel $P_{n}^{(\alpha,\beta)}(\cos\theta)$, 
   where $\theta\in \mathbb{T}$.  
Then for the following cases:
\begin{enumerate}
    \item $\alpha > \frac{1}{2}$,
    \item $0 \le \alpha < \frac{1}{2}$,
    \item $\alpha = \beta = \frac{1}{2}$,
\end{enumerate}
the operator $T_n^{\alpha,\beta}$ satisfies
\begin{align}\label{eq: Tn sharp bound}
    \|T_n^{\alpha,\beta}\|_{L^{p'}(\mathbb{T}) \to L^p(\mathbb{T})} \lesssim A(\alpha, p/2, n), \quad \text{for all } p \ge 2.
\end{align}

\end{theorem}

In particular, the case $\alpha=\beta=\tfrac12$ corresponds to the three-dimensional factors $\mathbb{S}^3$ or $\mathbb{RP}^3$ in the product manifold $M$. For this case, the sharp estimate \eqref{eq: Tn sharp bound} at the endpoint $p=2$ mirrors the log-loss removal of \eqref{eq: logloss} for geodesics on three-dimensional compact manifolds established by Chen--Sogge \cite{CS14}. In our setting of product manifolds containing three-dimensional factors, another key ingredient that enables the use of \eqref{eq: Tn sharp bound} for eigenfunction restriction to arbitrary submanifolds of maximal flats is the positivity of Fourier coefficients of spherical functions, as in Lemma \ref{lem: Phi_n=sum}. This positivity helps preserve oscillation in the convolution kernel and effectively reduces the study of arbitrary submanifolds to those of product type; see \eqref{eq: K=sum}.

The rest of the paper is organized as follows. In Section \ref{sec: proof of thm prod}, we prove the estimates in Theorem \ref{thm: prod} by first establishing restriction estimates for the joint eigenfunctions of the Laplace--Beltrami operators on each factor $M_i$, using Theorem \ref{thm: jac}. In Section \ref{sec: proof of thm jac}, we prove Theorem \ref{thm: jac}. In Section \ref{sec: sharpness}, we show the sharpness of \eqref{eq: no loss} in  Theorem \ref{thm: prod}.

\subsection*{Notation} Throughout the paper, we employ the standard notation $A \lesssim B$ to indicate that $A \leq CB$ for some constant $C > 0$.  If $A \lesssim B$ and $B \lesssim A$, we write $A \simeq B$. 
We write $A \lesssim_{\delta} B$ when $A \leq CB$ for some constant $C > 0$ depending on $\delta$. 
$A\ll B$ means there is a sufficiently small constant $c>0$ such that $A\leq cB$. 
We also write $A \sim B$ to mean that ${A}/{B} \to 1$.

\subsection*{Acknowledgments} The author thanks the referee for their valuable comments, which improved the paper. The author also thanks Xiaocheng Li and Cheng Zhang for helpful discussions.

\section{Proof of Theorem \ref{thm: prod}}\label{sec: proof of thm prod}
\subsection{Spherical functions}
Let $M$ be a CROSS of dimension $d$. 
The spherical functions $\Phi_n$ ($n\in\mathbb{Z}_{\geq 0}$) on $M$, when restricted to a maximal torus (large circle) $\mathbb{T}\cong\mathbb{R}/2\pi\mathbb{Z}$ of $M$, 
can be expressed by Jacobi polynomials as follows: 
\begin{align}\label{eq: Phi_n}
  \Phi_n(\theta)= \binom{n+\alpha}{n}^{-1}P_n^{(\alpha,\beta)}(\cos\theta),\qquad \theta\in\mathbb{T},
\end{align}
where $\alpha,\beta\in\frac12\cdot\mathbb{Z}$ such that 
\begin{align}\label{eq: alpha value}
    0\leq \beta\leq\alpha=\frac{d-2}{2}.
\end{align}
In particular, we have picked the standard normalization such that $\Phi_n(0)=1$. 
For the above formula, we refer to Theorem 4.5 in Chapter V of \cite{Hel00}, where the $\Phi_n$ are expressed in terms of hypergeometric functions, and to formula (4.21.2) in \cite{Sze75}, where these hypergeometric functions are rewritten as Jacobi polynomials. A specific case is when $d=3\Leftrightarrow \alpha=1/2$, corresponding to $\mathbb{S}^3$ or $\mathbb{RP}^3$, for which we also have 
$\beta=1/2$.

We will also use the following convenient formula, which expresses spherical functions as positive linear combinations of exponentials.

\begin{lemma}[Proposition 9.4 of Chapter III of \cite{Hel08}]\label{lem: Phi_n=sum}
For $n\in\mathbb{Z}_{\geq 0}$, we have 
    $$\Phi_n(\theta)=\sum_{j=0}^{q_n} c_{n,j} e^{im_{n,j}\theta},\qquad \theta\in\mathbb{T},$$
    where $c_{n,j}\geq 0$ and $m_{n,j}\in\mathbb{Z}$. In particular, $\sum_{j=0}^{q_n} c_{n,j}=\Phi_n(0)=1$. 
\end{lemma}

Let $k(n)$ denote the dimension of the spherical representation of the isometry group $U$ of $M$ associated to $\Phi_{n}$. We need the following standard fact about $k(n)$:

\begin{lemma}\label{lem: k(n) bound}
   $k(n)$ is a polynomial in $n$ of degree $d-1$, so that 
\begin{align}
k(n)\simeq (n+1)^{d-1}, \qquad n\in\mathbb{Z}_{\geq 0}. 
\end{align}
\end{lemma}
\begin{proof}
    This is a consequence of the Weyl dimension formula combined with the characterization of the highest weight of spherical representations (Theorem 4.1 of Chapter V of \cite{Hel00}). 
\end{proof}

\subsection{Eigenfunctions}
Let $M=M_1\times\cdots\times M_r$ be a product of CROSSs. 
Each \(M_i\), being a symmetric space of compact type, may be written as a homogeneous quotient \(M_i\cong U_i/K_i\).  Using the Haar measure \(du_i\) on the compact group \(U_i\), we have the convolution of functions \(f,g\) on \(M_i\) by
\[
(f*g)(x_i)=\int_{U_i} f(u_iK_i)\,g(u_i^{-1}x_i)\,du_i,\qquad x_i\in M_i.
\]
These convolutions on each \( M_i \) together induce the convolution on the product  $M$.

For $f\in L^2(M)$, following Proposition 9.1 of Chapter III of \cite{Hel00}, we have the $L^2$ decomposition 
\begin{align}\label{eq: f=sum}
    f=\sum_{n_i\in\mathbb{Z}_{\geq 0},\ i=1,\ldots,r}P_{n_1,\ldots,n_r}f,
\end{align}
where $P_{n_1,\ldots,n_r}$ is the projection onto the space of \textit{joint eigenfunctions} of the Laplace--Beltrami operators $\Delta_i$ on $M_i$, such that 
$$\Delta_i P_{n_1,\ldots,n_r}f=-(n_i^2+a_in_i)P_{n_1,\ldots,n_r}f.$$
Here each $a_i\in\mathbb{Z}_{>0}$ is a constant associated to $M_i$. Importantly, spherical functions are such joint eigenfunctions, and these projection operators are simply convolution with spherical functions: 
\begin{align}\label{eq: joint projection}
    P_{n_1,\ldots,n_r}f=f*\prod_{i=1}^rk_i(n_i)\Phi_{i,n_i}, 
\end{align}
where $\Phi_{i,n_i}$ ($n_i\in\mathbb{Z}_{\geq 0}$) are the spherical functions on $M_i$, and $k_i(n_i)$ denotes the dimension of the corresponding spherical representation. In particular, we have 
\begin{align}\label{eq: spherical convolution}
\prod_{i=1}^rk_i(n_i)\Phi_{i,n_i}*\prod_{i=1}^rk_i(n_i)\Phi_{i,n_i}=\prod_{i=1}^rk_i(n_i)\Phi_{i,n_i}.
\end{align}
We establish the following restriction bounds of joint eigenfunctions on products of CROSSs. 
\begin{theorem}\label{thm: joint}
Let $S$ be a $k$-dimensional submanifold of a maximal flat of $M$, $k=0,1,\ldots,r$. 
Suppose that $d_1\leq d_2\leq\cdots\leq d_r$. Let 
$$\tau(d,p):=\left\{\begin{array}{ll}
-\frac1p, & p\geq\frac{4}{d-1},\\
-\frac{d-1}{4}, & 0<p\leq\frac{4}{d-1}.
\end{array}\right.$$
Let $f$ be a joint eigenfunction of all the $\Delta_i$, $i=1,\ldots,r$, such that 
$$\Delta_i f=-(n_i+a_in_i)^2f, \qquad\text{for some } n_i\in\mathbb{Z}_{\geq 0}.$$
Let $N^2=\sum_{i=1}^r n_i^2+a_in_i >1$.  Then  
\begin{align}\label{eq: joint}
    \|f\|_{L^p(S)}\lesssim  N^{\frac{d-r}{2}+\sum_{i=1}^k\tau(d_i,p) } \|f\|_{L^2(M)}, \qquad\text{for all }p\geq 2.
\end{align} 
\end{theorem}

\begin{proof}
A maximal flat $E$ in $M$ is of the form 
$$E=\mathbb{T}_1\times\cdots\times\mathbb{T}_r,$$
where $\mathbb{T}_i$ is a maximal torus of $M_i$. Let $\theta_1,\ldots,\theta_r$ denote the corresponding angular variables on $E$. 
By passing to a finite open cover of $S$, 
we may assume that there exist $1\leq i_1<\cdots<i_k\leq r$ and an open subset 
$V$ of $\mathbb{T}_{i_1}\times\cdots\times \mathbb{T}_{i_k}$ 
such that the angular variables $\theta_{i_1},\ldots,\theta_{i_k}$ among $\theta_1,\ldots,\theta_r$ provide coordinates on $S$. 
Denote $\underline{\theta}:=(\theta_{i_1},\ldots,\theta_{i_k})$. Then $S$ is parametrized by the map
$$ V\ni\underline{\theta}\mapsto (\theta_1(\underline{\theta}),\ldots,\theta_r(\underline{\theta}))\in S\subset E.$$
Proving the bound \eqref{eq: joint} for joint eigenfunctions is equivalent to proving the $L^2(M)\to L^p(S)$ bound for the projection operators $P_{n_1,\ldots,n_r}$. 
By a $TT^*$ argument, and using \eqref{eq: joint projection}, \eqref{eq: spherical convolution} and the fact that the spherical functions are real-valued (as ensured by \eqref{eq: Phi_n}), 
it suffices to demonstrate the 
$N^{{d-r}+2\sum_{i=1}^k\tau(d_i,p)}$ bound for the $L^{p'}(V,J(\underline{\theta})\ d\underline{\theta})\to L^p(V,J(\underline{\theta})\ d\underline{\theta})$ norm of the operator 
$$P_{n_1,\ldots,n_r}P_{n_1,\ldots,n_r}^*g(\underline{\theta})=\int_V g(\underline{\theta}')\mathscr{K}(\underline{\theta}',\underline{\theta})J(\underline{\theta}')\ d\underline{\theta}',$$
where 
$$\mathscr{K}(\underline{\theta}',\underline{\theta})=\prod_{i=1}^rk_i(n_i)\Phi_{i,n_i}(\theta_i(\underline{\theta})-\theta_i(\underline{\theta}')),$$ 
and $J(\underline{\theta}')\ d\underline{\theta}'$ is the Riemannian measure on $S$, with $|J(\underline{\theta}')|\simeq 1$.

For $i\neq i_l$, $l=1,\ldots,k$, we use Lemma \ref{lem: Phi_n=sum} to write 
$$\Phi_{i,n_i}(\theta_i)=\sum_{j_i=0}^{q_{i,n_i}}c_{i,n_i,j_i}e^{im_{i,n_i,j_i}\theta_i},\qquad \theta_i\in \mathbb{T}_i,$$
where  $m_{i,n_i,j_i}\in\mathbb{Z}$, and 
\begin{align}\label{eq: c_i}
    c_{i,n_i,j_i}\geq 0, \qquad \sum_{j=0}^{q_{i,n_i}}c_{i,n_i,j_i}=1.
\end{align}
Then we may rewrite the kernel $\mathscr{K}(\underline{\theta}',\underline{\theta})$ and obtain
\begin{align}\label{eq: K=sum}
    \mathscr{K}(\underline{\theta}',\underline{\theta})=&\left(\prod_{i\notin\{i_1,\ldots,i_k\}}\sum_{j_i=0}^{q_{i,n_i}}c_{i,n_i,j_i}e^{im_{i,n_i,j_i}(\theta_i(\underline{\theta})-\theta_i(\underline{\theta}'))}\right) \widetilde{\mathscr{K}}(\underline{\theta}',\underline{\theta}),
\end{align}
where 
\begin{align*}
    \widetilde{\mathscr{K}}(\underline{\theta}',\underline{\theta})&=\prod_{i=1}^rk_i(n_i)\cdot \prod_{l=1}^k\Phi_{i_l,n_{i_l}}(\theta_{i_l}-\theta_{i_l}')\\
&=\prod_{i=1}^rk_i(n_i)\cdot \prod_{l=1}^k \binom{n_{i_l}+\alpha_{i_l}}{n_{i_l}}^{-1}P_{n_{i_l}}^{(\alpha_{i_l},\beta_{i_l})}(\cos(\theta_{i_l}-\theta_{i_l}')),
\end{align*}
with $0\leq\beta_{i_l}\leq \alpha_{i_l}=(d_{i_l}-2)/2$. 

Applying \eqref{eq: c_i}, \eqref{eq: K=sum}, and $|J(\underline{\theta}')|\simeq 1$, it suffices to get the same bound for the $L^{p'}(V,\ d\underline{\theta})\to L^p(V,\ d\underline{\theta})$ norm of the operator 
\begin{align*}
    \mathcal{A}: g\mapsto \int_V g(\underline{\theta}')\widetilde{\mathscr{K}}(\underline{\theta}',\underline{\theta})\ d\underline{\theta}'.
\end{align*}
Furthermore, extending $g$ by zero outside of $V$ in $\mathbb{T}^k:=\mathbb{T}_{i_1}\times\cdots\times\mathbb{T}_{i_k}$, it suffices to 
get the same bound for the $L^{p'}(\mathbb{T}^k,\ d\underline{\theta})\to L^p(\mathbb{T}^k,\ d\underline{\theta})$ norm of the operator 
\begin{align*}
    \widetilde{\mathcal{A}}: g\mapsto \int_{\mathbb{T}^k} g(\underline{\theta}')\widetilde{\mathscr{K}}(\underline{\theta}',\underline{\theta})\ d\underline{\theta}',
\end{align*}
as $\|\mathcal{A}\|_{L^{p'}(V,\ d\underline{\theta})\to L^p(V,\ d\underline{\theta})}\leq \|\widetilde{\mathcal{A}}\|_{L^{p'}(\mathbb{T}^k,\ d\underline{\theta})\to L^p(\mathbb{T}^k,\ d\underline{\theta})}$. 
Then the above operator is convolution on $\mathbb{T}^k$ with a product of Jacobi polynomials. 
First observe that by the standard asymptotics 
\begin{align}\label{eq: binom asymp}
    \binom{n+\alpha}{n} \sim \frac{n^{\alpha}}{\Gamma(\alpha+1)}, \qquad n \to \infty,
\end{align}
and Lemma \ref{lem: k(n) bound}, we have
\begin{align}\label{eq: final 1}
    \prod_{i=1}^rk_i(n_i)\cdot \prod_{l=1}^k \binom{n_{i_l}+\alpha_{i_l}}{n_{i_l}}^{-1}&\lesssim \prod_{i=1}^r(n_i+1)^{d_i-1}\prod_{l=1}^k
    (n_{i_l}+1)^{-\alpha_{i_l}}.
\end{align}
Then noting that $p\geq p'$ as $p\geq 2$, we apply the Minkowski inequality along with Theorem \ref{thm: jac}, to obtain 
\begin{align}\label{eq: final 2}
\left\|f*    \prod_{l=1}^k P_{n_{i_l}}^{(\alpha_{i_l},\beta_{i_l})}(\cos(\theta_{i_l}))\right\|_{L^p(\mathbb{T}^k)}
&\lesssim \prod_{l=1}^k \|T_{n_{i_l}}^{\alpha_{i_l},\beta_{i_l}}\|_{L^{p'}(\mathbb{T})\to L^{p}(\mathbb{T})}\|f\|_{L^{p'}(\mathbb{T}^k)}\nonumber\\
&\lesssim \prod_{l=1}^k(n_{i_l}+1)^{\alpha_{i_l}+2\tau(d_{i_l},p)}\|f\|_{L^{p'}(\mathbb{T}^k)}. 
\end{align}
Here in the last inequality we used $\alpha_{i_l}=(d_{i_l}-2)/2$. Combining \eqref{eq: final 1} and \eqref{eq: final 2} as well as the facts that $\tau(d,p)\geq -(d-1)/{4}$ and $n_i\lesssim N$, 
we conclude that 
$$\|\widetilde{\mathcal{A}}\|_{L^{p'}(\mathbb{T}^k)\to L^p(\mathbb{T}^k)}\lesssim N^{{d-r}+2\sum_{l=1}^k\tau(d_{i_l},p)}\leq N^{{d-r}+2\sum_{i=1}^k\tau(d_{i},p)}.$$
Here in the last inequality we applied $d_1\leq \cdots\leq d_r$. The proof is now complete.

\end{proof}

We are ready to finish the proof of Theorem \ref{thm: prod}.

\begin{proof}[Proof of Theorem \ref{thm: prod}]
Suppose $\Delta f=-N^2f$, $N>1$. 
By \eqref{eq: f=sum}, and a standard counting estimate \cite{Gro85}, we get
\begin{align*}
    |f|&\leq \left|\left\{(n_1,\ldots,n_r)\in\mathbb{Z}^r_{\geq 0}: \sum_{i=1}^r n_i^2+a_in_i=N^2\right\}\right|^{\frac{1}{2}}\|P_{n_1,\ldots,n_r}f\|_{l^2_{n_1,\ldots,n_r}}\\
    &\lesssim N^{\frac{r}{2}-1+\varepsilon}\|P_{n_1,\ldots,n_r}f\|_{l^2_{n_1,\ldots,n_r}},
\end{align*}
which implies by an application of the Minkowski inequality, that for $p\geq 2$, 
$$\|f\|_{L^p(M)}\lesssim N^{\frac{r}{2}-1+\varepsilon}\|\|P_{n_1,\ldots,n_r}f\|_{L^p(M)}\|_{l^2_{n_1,\ldots,n_r}},$$
and the $\varepsilon$-loss may be removed when $r\geq 5$. 

The estimates \eqref{eq: general} and \eqref{eq: no loss} of Theorem \ref{thm: prod} are now consequences of Theorem \ref{thm: joint}. We leave the proof of sharpness of \eqref{eq: no loss} to Section~\ref{sec: sharpness}.
\end{proof}

\section{Proof of Theorem \ref{thm: jac}}\label{sec: proof of thm jac}

In addition to \eqref{eq: A(delta,p,n)}, it will also be convenient to denote 
$$\widetilde{A}(\delta,p,n):=\left\{\begin{array}{lll}
      (n+1)^{\delta-\frac{1}{p}},  & p>\frac{1}{\delta+\frac12}, \\
        (n+1)^{-\frac{1}{2} } \log^{\delta+\frac12} (n+2), & p=\frac{1}{\delta+\frac12}, \\
         (n+1)^{-\frac{1}{2}},  & 0<p<\frac{1}{\delta+\frac12}.
   \end{array}\right.   
   $$
We first demonstrate Theorem \ref{thm: jac} except for the kink point: 
\begin{proposition}\label{prop}
For all the three cases in Theorem \ref{thm: jac}, we have 
\[
\|T_n^{\alpha,\beta}\|_{L^{p'}(\mathbb{T}) \to L^p(\mathbb{T})} \lesssim \widetilde{A}(\alpha, p/2, n), \quad \text{for all } p \ge 2.
\]  
\end{proposition}
\begin{proof}

We decompose the convolution kernel $P^{\alpha,\beta}_n(\cos\theta)$ according to the value of $\theta$, and it suffices to restrict attention to $\theta\in[0,\pi]$. Let $c$ be a fixed positive constant. Denote 
$$\widetilde{n}:=n+\frac{\alpha+\beta+1}{2}$$
and 
$$\gamma:=-\frac\pi{2}\left(\alpha+\frac12\right).$$ 

\textbf{Case 1.} $c(n+1)^{-1}\leq \theta\leq \pi-c(n+1)^{-1}$. We use (8.21.18) of \cite{Sze75} as follows:
\begin{align*}
    &P_{n}^{(\alpha,\beta)}(\cos\theta)=\\
    &\pi^{-\frac{1}{2}}n^{-\frac{1}{2}}\left(\sin\frac{\theta}{2}\right)^{-\alpha-\frac12}\left(\cos\frac{\theta}{2}\right)^{-\beta-\frac12}
    [\cos(\widetilde{n}\theta+\gamma)+(n\sin\theta)^{-1}O(1)],\qquad n\to\infty. 
\end{align*}

This implies  
\begin{align}\label{eq: bound for intermediate theta}
    |P_{n}^{(\alpha,\beta)}(\cos\theta)|\lesssim (n+1)^{-\frac{1}{2}}\left(\sin\frac{\theta}{2}\right)^{-\alpha-\frac12}\left(\cos\frac{\theta}{2}\right)^{-\beta-\frac12}.
\end{align}
    An elementary calculation then gives 
    \begin{align*}
         \|P_{n}^{(\alpha,\beta)}(\cos\theta)\|_{L^{p}_\theta([c(n+1)^{-1},\pi/2])}
   \lesssim \widetilde{A}(\alpha,{p},n),
    \end{align*}
and 
\begin{align}\label{eq: Pn beta}
    \|P_{n}^{(\alpha,\beta)}(\cos\theta)\|_{L^{p}_\theta([\pi/2,\pi-c(n+1)^{-1}])}
   \lesssim \widetilde{A}(\beta,{p},n).
\end{align} 
Noting $\widetilde{A}(\beta,{p},n)\leq \widetilde{A}(\alpha,{p},n)$, the above two estimates yield the desired bound of the convolution operator with the kernel $P^{\alpha,\beta}_n(\cos\theta)$ restricted to the interval $[c(n+1)^{-1},\pi-c(n+1)^{-1}]$, via an application of Young's convolution inequality. 
   
\textbf{Case 2.} $0<\theta\leq c(n+1)^{-1}$ or $\pi-c(n+1)^{-1}\leq\theta<\pi$. For the former, we use (8.21.17) of \cite{Sze75} as follows: 
\begin{align}
    &\left(\sin\frac{\theta}{2}\right)^{\alpha}\left(\cos\frac{\theta}{2}\right)^{\beta}P_n^{(\alpha,\beta)}(\cos\theta)\nonumber \\
=&\widetilde{n}^{-\alpha}\frac{\Gamma(n+\alpha+1)}{n!}\left(\frac{\theta}{\sin\theta}\right)^\frac12 J_\alpha(\widetilde{n}\theta)+ \theta^{\alpha+2}O(n^\alpha), \qquad n\to\infty.\label{eq: small theta}
\end{align}
Here $J_\alpha$ stands for the Bessel function of the first kind of order $\alpha$, as defined in (1.71.1) of \cite{Sze75}. By (1.71.10) of \cite{Sze75}, we have the asymptotics, 
$$J_\alpha(x)\sim x^\alpha, \qquad x\to 0+.$$
We also need the classical asymptotics for Gamma functions:
$$\frac{\Gamma(n+\alpha+1)}{n!}\sim n^{\alpha}, \qquad n\to\infty.$$
Apply the above two asymptotics to \eqref{eq: small theta}, we get 
$$|P_{n}^{(\alpha,\beta)}(\cos\theta)|\lesssim (n+1)^\alpha,\qquad 0<\theta<c(n+1)^{-1}.$$
This implies 
\begin{align}\label{eq: alpha-1/p}
    \|P_{n}^{(\alpha,\beta)}(\cos\theta)\|_{L^p_\theta((0,c(n+1)^{-1}])}
   \lesssim (n+1)^{\alpha-\frac1p},\qquad \text{for all }p>0.
\end{align}
A similar approach works for the other case $\pi-c(n+1)^{-1}\leq\theta<\pi$: combining (8.21.17) and (4.1.3) of \cite{Sze75}, we get 
\begin{align*}
    &\left(\sin\frac{\theta}{2}\right)^{\alpha}\left(\cos\frac{\theta}{2}\right)^{\beta}P_n^{(\alpha,\beta)}(\cos\theta)\\
=&(-1)^n \widetilde{n}^{-\beta}\frac{\Gamma(n+\beta+1)}{n!}\left(\frac{\pi-\theta}{\sin\theta}\right)^{\frac12} J_\beta(\widetilde{n}(\pi-\theta))+ (\pi-\theta)^{\beta+2}O(n^\beta),\qquad n\to\infty,
\end{align*}
which implies, in a similar manner as above, that 
\begin{align}\label{eq: beta - 1/p}
    \|P_{n}^{(\alpha,\beta)}(\cos\theta)\|_{L^p_\theta([\pi-c(n+1)^{-1},\pi))}
   \lesssim (n+1)^{\beta-\frac1p}\leq (n+1)^{\alpha-\frac1p}, \qquad \text{for all }p>0. 
\end{align}
Again by an application of Young's inequality, \eqref{eq: alpha-1/p} and \eqref{eq: beta - 1/p} together yield the desired bound of the convolution operator with the kernel $P_{n}^{(\alpha,\beta)}(\cos\theta)$ restricted to the intervals $(0,c(n+1)^{-1}]$ and $[\pi-c(n+1)^{-1},\pi)$. This completes the proof. 
\end{proof}
\begin{remark}
    The above proof yields the sharp $L^p$ estimates of Jacobi polynomials: for $0\leq\beta\leq\alpha$, we have 
    $$\|P_{n}^{(\alpha,\beta)}(\cos\theta)\|_{L^p_\theta([0,2\pi])}\lesssim\widetilde{A}(\alpha,p,n).$$
    In particular, 
    \begin{align}\label{eq: Linfty of Jacobi}
        \|P_{n}^{(\alpha,\beta)}\|_{L^\infty([-1,1])}\lesssim (n+1)^\alpha. 
    \end{align}
\end{remark}

Finally, we eliminate the logarithmic factor present in Proposition \ref{prop} for the cases relevant to Theorem \ref{thm: jac}.

\begin{proof}[Proof of Theorem \ref{thm: jac}]
Note that Proposition \ref{prop} already covers the case $\alpha>1/2$ of Theorem \ref{thm: jac}. For the case $0\leq\beta\leq \alpha<1/2$, it suffices to prove the desired bound 
$$\|T^{(\alpha,\beta)}_n\|_{L^{p'}(\mathbb{T})\to L^p(\mathbb{T})}\lesssim (n+1)^{\frac12}$$ at the kink point $p={4}/{(2\alpha+1)}>2$. To accomplish this, we adjust the proof of Proposition~\ref{prop}: 
the only change needed is to apply the Hardy--Littlewood--Sobolev inequality 
in place of Young's convolution inequality to \eqref{eq: bound for intermediate theta} 
over the interval $[c(n+1)^{-1},\pi/2]$, as well as over the interval 
$[\pi/2,\pi-c(n+1)^{-1}]$ if $\beta=\alpha$; if $\beta<\alpha$, the original application of Young's inequality over the interval 
$[\pi/2,\pi-c(n+1)^{-1}]$ via \eqref{eq: Pn beta} remains sufficient. 
For the remaining case $\alpha=\beta=1/2$, we use the explicit formula (4.1.7) of \cite{Sze75},
$$P^{(\frac12,\frac12)}_n(\cos\theta)=\binom{n+\frac12}{n}\frac{\sin (n+1)\theta}{(n+1)\sin \theta}.$$
Noting \eqref{eq: binom asymp}, the desired bound for $\|T^{(\frac12,\frac12)}_n\|_{L^{2}(\mathbb{T})\to L^2(\mathbb{T})}$ results from the $L^2$-boundedness of the convolution operator with the Dirichlet kernel. The proof is now complete.
\end{proof}

\section{Sharpness of \eqref{eq: no loss}}\label{sec: sharpness}
Fix an origin $e=(e_1,\ldots,e_r)$ of $M=M_1\times\cdots\times M_r$, and take $S$ to be any open subset of $\mathbb{T}_1 \times \cdots \times \mathbb{T}_k\times\{e_{k+1}\}\times\cdots\times \{e_{r}\}$ containing $e$. 
For $N>1$, take 
$$\Lambda:=\{(n_1,\ldots,n_r)\in\mathbb{Z}_{\geq 0}^r: n_1\geq \cdots\geq n_r\geq \frac{n_1}{2},\ \sum_{i=1}^rn_i^2+a_in_i=N^2\},$$
and then take 
\begin{align}\label{eq: maximizer}
    f(\theta_1,\ldots,\theta_r)=\sum_{(n_1,\ldots,n_r)\in \Lambda}\prod_{i=1}^r\sqrt{k_i(n_i)}\, \Phi_{i,n_i}(\theta_i).
\end{align}
As the spherical functions $\Phi_{i,n_i}$ are matrix entries of the associated irreducible spherical representations expressed via unit vectors (see Theorem 4.3 of Chapter V of \cite{Hel00}), the Schur orthogonality relations yield orthogonality between $\Phi_{i,n_i}$ for different $n_i$, and 
$$\|\sqrt{k_i(n_i)}\, \Phi_{i,n_i}\|_{L^2(M_i)}=1,$$
which together imply 
\begin{align}\label{eq: f L2}
    \|f\|_{L^2(M)}=|\Lambda|^{\frac12}.
\end{align}
We need the following derivative estimate for spherical functions. 
\begin{lemma}
  For spherical functions $\Phi_n(\theta)$ ($n\in\mathbb{Z}_{\geq 0}$) on a CROSS, we have the pointwise estimate 
    $$|\Phi_{n}'(\theta)|\lesssim (n+1)^2\sin\theta, \qquad\theta\in\mathbb{T}.$$
    In particular, this implies that 
    \begin{align}\label{eq: Phi_n close to 1}
        |\Phi_n(\theta)-1|\ll 1,\qquad\text{if } |\theta|\ll \frac{1}{n+1}. 
    \end{align}
\end{lemma}
\begin{proof}

    By (4.21.7) of \cite{Sze75} and \eqref{eq: Phi_n}, we have 
    $$\Phi_n'(\theta)=-\frac{\sin\theta}{2}(n+\alpha+\beta+1)\binom{n+\alpha}{n}^{-1}P^{(\alpha+1,\beta+1)}_{n-1}(\cos\theta).$$
   Noting \eqref{eq: binom asymp} and \eqref{eq: Linfty of Jacobi}, the result follows. 
\end{proof}

Note that  $n_i\sim N$ for all $i$ whenever $(n_1,\ldots,n_r)\in\Lambda$. 
By \eqref{eq: Phi_n close to 1} and Lemma \ref{lem: k(n) bound}, if $|\theta_i|\ll 1/N$ for all $i$, then for the chosen $f$ in \eqref{eq: maximizer}, we have 
$$|f(\theta_1,\ldots,\theta_r)|\gtrsim \sum_{(n_1,\ldots,n_r)\in \Lambda}\prod_{i=1}^r\sqrt{k_i(n_i)}\gtrsim N^{\frac{d-r}{2}}|\Lambda|.$$
The above estimate further implies that 
$$\|f\|_{L^p(S)}\gtrsim N^{\frac{d-r}{2}-\frac{k}{p}}|\Lambda|,$$
which together with \eqref{eq: f L2} yields
$$\frac{\|f\|_{L^p(S)}}{\|f\|_{L^2(M)}}\gtrsim N^{\frac{d-r}{2}-\frac{k}{p}}|\Lambda|^{\frac12}.$$
When $r\geq 5$, by the equidistribution of lattice points on $r$-dimensional spheres as established by Pommerenke \cite{Pom59}, it holds 
$|\Lambda| \gtrsim N^{r-2}$ whenever $\Lambda\neq\varnothing$, which yields the sharpness of \eqref{eq: no loss}.

\bibliographystyle{acm}

\bibliography{arxiv}

\end{document}